\documentclass[12pt]{amsart}

\usepackage[utf8]{inputenc}
\usepackage[T1]{fontenc}	
\usepackage[french,english]{babel}	

\usepackage{lmodern}			
\usepackage{newtxtext}

\usepackage[top=3cm, bottom=3cm, left=3.6cm, right=3.6cm]{geometry}
\usepackage{graphicx}	

\usepackage[a4paper,plainpages=false,colorlinks,linkcolor=bleuFonce,citecolor=rougeFonce,urlcolor=vertFonce,breaklinks]{hyperref}

\usepackage{placeins} 
\usepackage{float} 

\usepackage{tikz}
\usetikzlibrary{patterns}

\usepackage[normalem]{ulem}

\usepackage{xcolor}
\definecolor{vertFonce}	{rgb}{0,0.5,0}
\definecolor{numLignes}	{rgb}{0.17,0.57,0.7}	
\definecolor{gris}		{rgb}{0.5,0.5,0.5}
\definecolor{grisFonce}	{rgb}{0.2,0.2,0.2}
\definecolor{orange}	{rgb}{1,0.65,0.31}		
\definecolor{orangeFonce}{rgb}{1,0.4,0}
\definecolor{bleuFonce}	{rgb}{0,0,0.4}
\definecolor{rougeFonce}{rgb}{0.3,0,0}
\definecolor{rougeWord}	{rgb}{0.5,0,0}
\definecolor{vertClair}	{rgb}{0.8,1,0.8}
\definecolor{rougeClair}{rgb}{1,0.5,0.5}
\definecolor{violet}	{rgb}{0.5,0,0.5}

\usepackage{pict2e}
\usepackage{multido}

\usepackage{amsfonts,amssymb,amsthm,amsmath}
\usepackage{amsaddr}	
\usepackage{dsfont}				
\usepackage{mathrsfs}
\usepackage{yfonts}
\usepackage[mathscr]{euscript}
\usepackage{cancel}
\usepackage{enumerate}

\usepackage{mathtools} 
 
\newtheorem{theorem}{Theorem}[section]

\newtheorem{prop}{Proposition}[section]

\newtheorem{remark}{Remark}[section]

\newenvironment{system}{%
	\equation\left\{\ \begin{aligned}%
}{%
	\end{aligned} \right. \endequation%
}
\newenvironment{system*}{%
	\equation\nonumber\left\{\ \begin{aligned}
}{%
	\end{aligned} \right. \endequation%
}

\usepackage{stmaryrd}
\SetSymbolFont{stmry}{bold}{U}{stmry}{m}{n}
\newcommand		{\subsetArrow}	{\mathrel{\ooalign{$\subset$\cr%
\hidewidth\raise-.087ex\hbox{$_\shortrightarrow\mkern-1.5mu$}\cr}}}
\newcommand		{\subsetarrow}	{\mathrel{\ooalign{$\subset$\cr%
\hidewidth\raise-1.45ex\hbox{$\vec{}\mkern6mu$}\cr}}}

\newcommand		{\N}		{\mathbb N}			
\newcommand		{\RR}		{\mathbb R}			
\newcommand		{\R}		{\RR}

\newcommand		{\Rd}		{\R^3}
\newcommand		{\Rdd}		{\R^{6}}

\newcommand		{\cP}		{\mathcal P}		
\renewcommand	{\L}		{\mathcal L}		
\renewcommand	{\P}		{\mathscr P}		

\newcommand		\sfm 		{\mathsf m}

\newcommand		{\lt}			{\left}				%
\newcommand		{\rt}			{\right}			%
\renewcommand	{\(}			{\lt(}
\renewcommand	{\)}			{\rt)}
\newcommand		{\bangle}[1]	{\lt\langle #1\rt\rangle}
\newcommand		{\weight}[1]	{\bangle{#1}}	

\newcommand		{\com}[1]		{\lt[{#1}\rt]}		

\newcommand		{\n}[1]			{\lt\lvert #1 \rt\rvert}
\newcommand		{\norm}[1]		{\lVert #1 \rVert}		
\newcommand		{\nrm}[1]		{\lt\lVert #1\rt\rVert}

\newcommand		{\Bnrm}[1]		{\Big\lVert #1\Big\rVert}
\newcommand		{\Nrm}[2]		{\nrm{#1}_{#2}}

\newcommand		{\BNrm}[2]		{\Bnrm{#1}_{#2}}

\renewcommand		{\d}		{\mathop{}\!\mathrm{d}}		

\newcommand			{\dpt}		{\partial_t}

\newcommand			{\dt}		{\frac{\d}{\d t}}	

\newcommand			{\Dx}		{\nabla_x}
\newcommand			{\Dv}		{\nabla_\xi}

\DeclareMathOperator{\sign}		{sign}				
\DeclareMathOperator{\tr}		{Tr}				
\DeclareMathOperator{\diag}		{diag}

\newcommand		{\Sign}[1]		{\sign\!\( #1 \)}	
\newcommand		{\Tr}[1]		{\tr\!\( #1 \)}		
\newcommand		{\Diag}[1]		{\diag\!\( #1 \)}

\newcommand		{\intd}			{\int_{\Rd}}
\newcommand		{\intdd}		{\int_{\Rdd}}
\newcommand		{\iintd}		{\iint_{\Rdd}}
\newcommand		{\iintdd}		{\iint_{\Rdd\times\Rdd}}

\newcommand		{\init}			{\mathrm{in}}

\newcommand		{\eps}			{\varepsilon}
\newcommand		{\Eps}			{\mathcal{E}}
\newcommand		{\cC}			{\mathcal{C}}

\usepackage{braket}

\newcommand		{\op}		{\boldsymbol{\rho}}	
\newcommand		{\opgam}	{\boldsymbol{\gamma}}	

\newcommand		{\Wh}		{W_{2,\hbar}}		
\newcommand		{\weyl}		{\op_\hbar^W}
\newcommand		{\wick}		{\widetilde{\op}_\hbar^W}
\newcommand		{\Weyl}[1]	{\weyl\!\!\(#1\)}
\newcommand		{\Wick}[1]	{\wick\!\!\(#1\)}

\newcommand		{\opc}		{\boldsymbol{c}}

\newcommand		{\opp}		{\boldsymbol{p}}

\newcommand		{\Dh}		{\boldsymbol{\nabla}}	
\newcommand		{\Dhx}[1]	{\Dh_{\!x} #1}			
\newcommand		{\Dhv}[1]	{\Dh_{\!\xi} #1}		

\title[\textsc{Uniqueness criteria and semiclassical analysis}]{\Large Uniqueness criteria for the Vlasov--Poisson system \\and applications to semiclassical analysis}

\author[\textsc{L.~Lafleche}]{\large\textsc{Laurent Lafleche}}
\address[L.~Lafleche]{Institut Camille Jordan, UMR 5208 CNRS \\\& Université Claude Bernard Lyon 1, France}
\email{lafleche@math.univ-lyon1.fr}

\author[\textsc{C.~Saffirio}]{\large\textsc{Chiara Saffirio}}
\address[C.~Saffirio]{Department of Mathematics and Computer Science,\\ University of Basel, 4051 Basel, Switzerland}
\email{chiara.saffirio@unibas.ch}


\begin{document}

\begin{abstract}
	We review some uniqueness criteria for the Vlasov--Poisson system, emerging as corollaries of stability estimates in strong or weak topologies, and show how they serve as a guideline to solve problems arising in semiclassical analysis. Different topologies allow to treat different classes of quantum states.
\end{abstract}

\begingroup
\def\uppercasenonmath#1{} 
\let\MakeUppercase\relax 
\maketitle
\thispagestyle{empty} 
\endgroup



\section{Introduction}

	We consider the Vlasov--Poisson equation 
	\begin{equation}\label{eq:VP}
		\dpt f + \xi\cdot\Dx f+E_f\cdot\nabla_\xi f=0,
	\end{equation}
	where $f=f(t,x,\xi)$ is a real-valued non-negative time-dependent function defined on the phase space $\Rdd$, $E_f(t,x)=(\nabla K*\rho_f)(t,x)$ is a self-induced force field generated by the spatial density $\rho_f(t,x) = \int f(t,x,\xi)\d\xi$. To simplify the presentation, in the rest of the paper we will focus on $d=3$ and 
	\begin{equation*}
		K(x)=\frac{\pm1}{\n{x}}.
	\end{equation*}
	In this setting, the well-posedness theory for the Cauchy problem associated with Equation~\eqref{eq:VP} has been established by Pfaffelmoser~\cite{pfaffelmoser_global_1992} and Lions and Perthame~\cite{lions_propagation_1991}, respectively proving global existence and uniqueness of classical solutions with compactly supported initial data, and existence of global weak solutions assuming that the initial data $f^\init$ are such that $f^\init\in L^1(\Rdd)\cap L^\infty(\Rdd)$ and $\int \n{\xi}^m f^\init\d x\d\xi$ is finite for some $m > 3$. A sufficient condition for uniqueness of measure-valued solutions with spatial density satisfying $\rho_f\in L^\infty([0,T],L^\infty(\Rd))$ was proven by Loeper in~\cite{loeper_uniqueness_2006}, and more recently generalized by Miot~\cite{miot_uniqueness_2016} including spatial densities such that $\sup_{p\geq 1} \frac{1}{p}\Nrm{\rho_f}{L^p}$ is finite, uniformly in time, and by Holding and Miot \cite{holding_uniqueness_2018} considering spatial densities in some Orlicz space. 
	
	The above mentioned uniqueness criteria deal with assumptions at positive time on the spatial density $\rho_f$. In this paper we collect three stability estimates for the Vlasov--Poisson equation leading to uniqueness criteria (non optimal in terms of conditions on $\rho_f$) in {\it weak-strong} sense. By weak-strong stability we indicate that some regularity conditions are assumed only on one of the two solutions involved in the stability estimate, thus allowing the other solution to satisfy only the minimal existence requirements.  We will show how such a weak-strong form in our estimates (cf.~Proposition~\ref{prop:weak-strong-L1} and Proposition~\ref{prop:weak-strong-L2}) can be used to deduce results in the context of the semiclassical limit from the Hartree equation (Equation~\eqref{eq:Hartree} below) towards the Vlasov--Poisson system~\eqref{eq:VP}. We will also look at the weak topology induced by the 2-Wasserstein distance (cf.~Section~\ref{subsec:opt-transport_classical}) and show how one can get a semiclassical estimate as an application of a stability result \`a la Loeper.  The weak topology allows to treat a class of quantum states that were not included in the analysis with the weak-strong method and that are relevant at zero temperature (so-called pure states), however it does not provide a weak-strong estimate. More precisely, we will review the results in \cite{lafleche_strong_2021, chong_l2_2022,lafleche_propagation_2019} and \cite{lafleche_propagation_2019,lafleche_global_2021}, giving a unified picture of the methods. 
			
	\subsection{Notations}For functions on the phase space of the form $f=f(x,\xi)$, we use the shorthand notation $L^p_xL^{q,r}_\xi = L^p(\Rd,L^{q,r}(\Rd))$, where $L^p(\Rd)$ and $L^{q,r}(\Rd)$ denote respectively the Lebesgue and Lorentz spaces. We also use the notation $\P = \P(\Rdd)$ for the set of probability measures on $\Rdd$. Let $\alpha\in\N^6$ be a multi-index with $\n{\alpha} = \sum_{i=1}^6\alpha_i$, and $\partial^{\alpha}=\partial^{\alpha_2}_{x_1}\partial_{\xi_1}^{\alpha_2}\dots\partial_{x_3}^{\alpha_5}\partial_{\xi_3}^{\alpha_6}$, then the weighted Sobolev norm in the Sobolev space $W^{k,p}_n$ can be defined by the formula
	\begin{equation*}
		\Nrm{f}{W^{k,p}_n} = \Big(\sum_{\n{\alpha}\leq k}\Nrm{\weight{\xi}^n\partial^\alpha f}{L^p}^2\Big)^\frac{1}{2}
	\end{equation*}
	with $\weight{z} = \sqrt{1+\n{z}^2}$, with the usual associated notations $H_n^k = W_n^{k,2}$ and $L^p_n = W^{0,p}_n$.
	
	Using the correspondence of classical and quantum mechanics and using the symbol $\{\cdot,\cdot\}$ and $\com{\cdot,\cdot}$ to denote respectively the Poisson brackets and the commutator, we recall that the quantum analogue of 
	\begin{equation*}
		\Dx f=\lt\{-\xi,f\rt\} \qquad \text{and} \qquad \Dv f = \lt\{x,f\rt\}
	\end{equation*}
	are respectively
	\begin{equation*}
		\quad\Dhx\op := \com{\nabla,\op}\ \qquad \text{and} \qquad \Dhv\op := \com{\frac{x}{i\hbar},\op},
	\end{equation*}
	where $\hbar=h/(2\pi)$ is the reduced Plank constant. \\
	The quantum analogue of Lebesgue norms and weighted Sobolev norms of $\op$, a bounded operator acting on $L^2(\Rd)$, are the quantum Lebesgue and Sobolev norms
	\begin{align*}
		\Nrm{\op}{\L^p} &= h^\frac{3}{p} \Tr{\n{\op}^p}^{\frac{1}{p}}
		\\
		\Nrm{\op}{\mathcal{W}^{k,p}_n} &= \Nrm{\sfm\op}{\L^p}+\Nrm{\sfm\Dhx\op}{\L^p}+\Nrm{\sfm\Dhv\op}{\L^p}
	\end{align*} 
	with $\n{\op} = \sqrt{\op^*\op}$ and $\sfm:=1+\n{\opp}^n$, with $\opp=-i\hbar\nabla$. When $p=\infty$, $\Nrm{\op}{\L^\infty}$ denotes the operator norm.

	To connect functions on the phase space and operators acting on $L^2(\Rd)$, we define the Weyl quantization as follows: to every $f\in L^1(\Rdd)$, the Weyl quantization of $f$ is the operator $\Weyl{f}$ with kernel
	\begin{equation*}
		\Weyl{f}(x,y) = \intd e^{-2i\pi (y-x)\cdot\xi}\, f\!\(\frac{x+y}{2},h\xi\)\d\xi.
	\end{equation*}
	Observe that, if $f$ is a probability density on the phase space (hence nonnegative), its Weyl quantization is not necessarily a non-negative operator. For this reason it is convenient to introduce an alternative quantization of the function $f$, called the T\"oplitz quantization, or the Wick quantization, or the Husimi quantization, given by
	\begin{equation*}
		\Wick{f} = \frac{1}{h^3}\intdd f(z) \ket{\psi_z}\bra{\psi_z}\d z, 
	\end{equation*}
	where 
	\begin{equation*}
		\psi_z(y) := \(\pi \hbar\)^{-3/4} e^{-\n{y-x}^2/(2\hbar)}\,e^{i y\cdot \xi/\hbar}.
	\end{equation*}
	It follows from this formula that $\Wick{f}$ is a positive operator whenever $f\ge 0$.
	
\noindent One can also define the inverse transformation 	of the Weyl quantization. Namely, let $\op$ be an operator with a regular kernel $\op(x,y)$, then
	\begin{equation*}
	f_{\op}(x,\xi)=\int_{\R^3} e^{-2\pi\,i\,\xi\cdot y/h}\op\(x+\frac{y}{2},x-\frac{y}{2}\)\,\d y
	\end{equation*}
	is the Wigner transform of the operator $\op$.
	
	If not otherwise stated, we will denote by $C$ a positive constant independent of $\hbar$ and will use that, for $a,b\geq 0$, $a \lesssim b$ if and only if there exists $C>0$ such that $a\leq C\,b$.

\section{\texorpdfstring{$L^1$}{L1} weak-strong uniqueness criterion}

	\noindent In this section we present a $L^1$ weak-strong stability estimate for the Vlasov--Poisson equation, where regularity is required only on one of the two solutions of~\eqref{eq:VP} and show an application of this criterion to the semiclassical limit from the Hartree equation to the Vlasov--Poisson system.
	
	\subsection{\texorpdfstring{$L^1$}{L1} weak-strong stability criterion for the Vlasov--Poisson equation}
		
	\begin{prop}\label{prop:weak-strong-L1}
		Let $f_1$ and $f_2$ be two positive solutions of the Vlasov--Poisson equation~\eqref{eq:VP} verifying  $f_1,\,f_2\in L^\infty([0,T],L^1(\Rdd))$, for some $T>0$. Then, if
		\begin{equation}\label{eq:regularity-f2}
			\lambda_{f_2}(t) := \Nrm{\nabla_\xi f_2}{L^{3,1}_xL^1_{\xi}}
		\end{equation}
		is finite for every $t\in [0,T]$, the following stability estimate holds
		\begin{align*}
			\Nrm{f_1-f_2}{L^1(\Rdd)} \leq \Nrm{f_1^\init-f_2^\init}{L^1(\Rdd)} e^{\Lambda_{f_2}(t)},
		\end{align*}
		where  $\Lambda_{f_2}(t) = C \int_0^t \lambda_{f_2}(s)\d s$ with $C = \Nrm{\nabla K}{L^{\frac{3}{2},\infty}}$.
	\end{prop}
	
	\begin{remark}
		Notice that the regularity condition~\eqref{eq:regularity-f2} on $f_2$ reads
		\begin{equation*}
			\intd\n{\nabla_\xi f_2}\d \xi \in L^{3,1}_x.
		\end{equation*}
		In particular, for each $t\in [0,T]$ and for some $\eps\in(0,2]$, it holds by real interpolation
		\begin{equation}\label{eq:regularity}
			\Nrm{\intd\n{\nabla_\xi f_2}\d \xi}{L^{3,1}_x} \leq  \Nrm{\intd\n{\nabla_\xi f_2}\d \xi}{L^{3+\eps}_x} + \Nrm{\intd\n{\nabla_\xi f_2}\d \xi}{L^{3-\eps}_x}.
		\end{equation}
		Indeed it is known that solutions of Vlasov--Poisson are regular (see e.g.~\cite{horst_global_1987,lions_propagation_1991,pfaffelmoser_global_1992,schaeffer_global_1991}). This is one of the advantages of the weak-strong method, that requires regularity only on one of the solutions.
	\end{remark}

	\begin{proof}[Proof of Proposition~\ref{prop:weak-strong-L1}]
		Let $f:= f_1-f_2$ and define for $k=1,2$, $\rho_k = \intd f_k\d \xi$, $\rho := \rho_1-\rho_2$ and $E_k = -\nabla K * \rho_k$. Then it holds
		\begin{align*}
			\dpt f + \xi\cdot\Dx f + E_1\cdot\Dv f = (E_2-E_1)\cdot\Dv f_2,
		\end{align*}
		whence
		\begin{align*}
			\dpt \Nrm{f}{L^1} &= \iintd \(\nabla K * \rho \cdot \Dv f_2\) \Sign{f} \d x\d \xi
			\\
			&= -\intd \rho\, \nabla K\,\dot{*} \(\intd \Sign{f} \Dv f_2\d \xi\)
		\end{align*}
		where the notation $\dot{*}$ indicates that we perform the scalar product of vectors inside the convolution. This implies
		\begin{equation}\label{eq:classical-L1}
			\dpt \Nrm{f}{L^1} \leq \Nrm{f}{L^1} \Nrm{\nabla K * \intd \n{\Dv f_2}\d\xi}{L^\infty}.
		\end{equation}
		By H\"older's inequality for Lorentz spaces (cf.~\cite[Formula~(2.7)]{hunt_lpq_1966}), since $\nabla K\in L^{\frac{3}{2},\infty}$, we have
		\begin{equation*}
			\Nrm{\nabla K * \Nrm{\nabla_\xi{f_2}}{L^1_\xi}}{L^\infty} \leq \sup_{z\in\Rd} \int_{\Rd} \n{\nabla K(z-\cdot)\Nrm{\nabla_\xi{f_2}}{L^1_\xi}} \leq C \Nrm{\nabla_\xi{f_2}}{L_x^{3,1}L^1_\xi}
		\end{equation*}
		and Gr\"onwall's Lemma concludes the proof.
	\end{proof}

\subsection{Applications to the limit from Hartree to Vlasov--Poisson: convergence in trace norm}

	In this section we consider the time-dependent Hartree equation 
	\begin{equation}\label{eq:Hartree}
		i\,\hbar\,\dpt\op = \com{H_{\op},\op}
	\end{equation}
	describing the time evolution of $\op=\op(t)$, the nonnegative and bounded density operator on $L^2(\R^3)$ such that $\Nrm{\op}{\L^1}=1$. The Hamiltonian is given by the sum of the Laplace operator and the multiplication operator $V_{\op}=K*\diag{(\op)}$
	\begin{equation}\label{eq:Hamiltonian}
		H_{\op}=-\frac{\hbar^2}{2}\Delta+V_{\op},
	\end{equation}
	where $\hbar = h/(2\pi)$ is the reduced Planck constant and 
	\begin{equation*}
		\Diag{\op}(x) := h^3 \,\op(x,x).
	\end{equation*}
	We recall the following result from~\cite{lafleche_strong_2021}.
	
	\begin{theorem}\cite[Theorem~1.1]{lafleche_strong_2021}\label{thm:LS-20}
		Let $f$ be a nonnegative solution to the Vlasov--Poisson equation~\eqref{eq:VP} and $\op$ a solution to the Hartree equation~\eqref{eq:Hartree} with initial data
		\begin{align*}
			f^\init &\in W_m^{\sigma+1,\infty}(\Rdd)\cap H_\sigma^{\sigma+1}(\Rdd)
			\\
			\op^\init &\in\L^1
		\end{align*} 
		with $m>3$ and $\sigma> m+6$. Then there exist two nonnegative real-valued continuous functions $\Lambda_{f}(t)$ and $C_{f}(t)$ independent of $\hbar$ such that
		\begin{equation*}
			\Nrm{\op-\Weyl{f}}{\L^1} \leq \(\Nrm{\op^\init-\Weyl{f^\init}}{\L^1} + C_f(t)\,\hbar\) e^{\Lambda_{f}(t)}.
		\end{equation*}
	\end{theorem}
	
	\begin{remark}
		Recalling that for $(y,w)\in\Rdd$
		\begin{equation*}
			\mathcal{F}[f_{\op}](y,w) = h^3\Tr{e^{-2i\pi\(y\cdot x+w\cdot \opp\)}\op}\,
		\end{equation*}
		where $x$ and $\opp$ are the position and momentum operator respectively and $\mathcal{F}[g]$ denotes the Fourier transform of the function $g$, we obtain 		
		\begin{equation*}
			\Nrm{\mathcal{F}[f_{\op}]-\mathcal{F}[f]}{L^\infty}\leq \Nrm{\op-\Weyl{f}}{\L^1}.
		\end{equation*}
		As a corollary of Theorem~\ref{thm:LS-20}, this gives convergence in $L^\infty$ for the Fourier transform of the difference between the Wigner transforms of the solution to the Hartree equation and the solution to the Vlasov--Poisson equation in the semiclassical limit, if the initial data are close for $\hbar\to 0$, providing an explicit rate of convergence, that is supposed to be  optimal in $\hbar$.	
		 This is a generalization in a stronger topology with explicit rate of the result obtained by Lions and Paul in~\cite{lions_sur_1993}. 
	\end{remark}
	
	The rest of this section is devoted to highlighting the analogy between the proof of Proposition~\ref{prop:weak-strong-L1} and the proof of Theorem~\ref{thm:LS-20}, aiming at conveying the idea that, given a weak-strong uniqueness stability estimate, there is a quite general method leading to the proof of the semiclassical limit in the correspondent topology for operators. 
	
	\subsubsection{Proof of Theorem~\ref{thm:LS-20} part I: reduction to stability estimates} Let $\Weyl{f}$ be the Weyl quantization of $f$ and $U_{t,s}$, $t>s$, be the two parameter semigroup generated by the Hartree Hamiltonian~\eqref{eq:Hartree}, i.e.
	\begin{equation*}
		i\,\hbar\,\partial_t U_{t,s}=H_{\op}(t)\,U_{t,s},\quad\quad U_{s,s}=1.
	\end{equation*}
	By conjugating the difference of $\op$ and $\Weyl{f}$ by $U_{t,s}$ and performing the time derivative we obtain
	\begin{equation*}
		i\,\partial_t U_{t,s}^*(\op-\Weyl{f})\,U_{t,s}=U_{t,s}^*\com{K*(\rho-\rho_f),\Weyl{f}}\,U_{t,s} +U_{t,s}^*B\,U_{t,s}
	\end{equation*}
	with $B_t$ the time-dependent operator with integral kernel
	\begin{equation}\label{eq:B-term}
		B(x,y)=\left((K*\rho_f)(x)-(K*\rho_f)(y)-(\nabla K*\rho_f)\!\left(\frac{x+y}{2}\right)\cdot(x-y)\right)\op_{\hbar}^W(x,y)
	\end{equation}
	and $\rho_f(x)=\diag(\Weyl{f})(x)$. By Duhamel's formula and using the unitarity of $U_{t,s}$ we get
	\begin{equation}\label{eq:duhamel}
		\begin{split}
		\Nrm{\op-\Weyl{f}}{\L^1}&\leq \Nrm{\op^\init-\Weyl{f^\init}}{\L^1} + \frac{1}{\hbar}\int_0^t\Nrm{B_s}{\L^1}\d s
		\\
		&\quad + \frac{1}{\hbar} \int_0^t\intd \n{\rho(x)-\rho_f(x)} \Nrm{\com{K(\cdot-x),\Weyl{f}}}{\L^1}\d x \d s.
		\end{split}
	\end{equation}	
	The structure of Equation~\eqref{eq:duhamel} is clear in light of the proof of Proposition~\ref{prop:weak-strong-L1}: the second line in Equation~\eqref{eq:duhamel} is the quantum analogue of Inequality~\eqref{eq:classical-L1} and has to be treated in the same spirit of a weak-strong stability estimate, whereas the second term in the right-hand side of the first line represents the error due to the fact that we are considering two different dynamics, namely the Hartree and the Vlasov ones. The error term is easily bounded in \cite[Proposition~4.4]{lafleche_strong_2021} by
	\begin{equation*}
		\Nrm{B}{\L^1} \leq C\,\hbar^2\Nrm{\rho_f}{L^1\cap H^\frac{3}{2}}\Nrm{\Dv^2 f}{H^{2n}_{2n}},\quad\quad n>\frac{3}{2}.
	\end{equation*}
	
	\subsubsection{Proof of Theorem~\ref{thm:LS-20} part II: the commutator inequality} To bound the second line in Equation~\eqref{eq:duhamel}, we proceed as in the proof of Proposition~\ref{prop:weak-strong-L1}, taking the $L^\infty$ norm in the variable $z$ in the commutator term. This yields 
	\begin{equation}\label{eq:commutator}
		\begin{split}
		\int_0^t\!\intd \n{\rho(x)-\rho_f(x)} & \Nrm{\com{K(\cdot-x),\Weyl{f}}}{\L^1}\d x \d s
		\\
		&\ \ \leq \sup_{x\in\Rd}\Nrm{\com{K(\cdot-x),\Weyl{f}}}{\L^1} \Nrm{\op-\Weyl{f}}{\L^1}.
		\end{split}
	\end{equation}
	As is the proof of Proposition~\ref{prop:weak-strong-L1}, we need to bound the commutator ${\com{K(\cdot-x),\Weyl{f}}}$ in $\L^1$, uniformly in $x$ and in $\hbar$. This is achieved in \cite[Theorem~4.1]{lafleche_strong_2021}, that we report here for completeness.
	\begin{prop}\cite[Theorem~4.1]{lafleche_strong_2021}\label{prop:commutator-inequality}
		Let $\op$ be a self-adjoint bounded operator and $K(x)=\frac{1}{\n{x}}$, then for any $\eps\in(0,2]$, there exists $C>0$ such that
		\begin{equation*}
			\Nrm{\com{K(\cdot-x),\op}}{\L^1}\leq\,C\,\hbar\,\(\Nrm{\diag(|\Dhv\op|)}{L^{3-\eps}}+\Nrm{\diag(|\Dhv\op|)}{L^{3+\eps}}\).
		\end{equation*}
	\end{prop}
	
	Applying Proposition~\ref{prop:commutator-inequality} to Equation~\eqref{eq:commutator}, we are left with bounding $\diag(|\Dhv\Weyl{f}|)$ in $L^p$, for $p\in\{3-\eps,3+\eps\}$. This is the quantum analogue of Inequality~\eqref{eq:regularity}. Mimicking the classical interpolation inequality
	\begin{equation*}
		\Nrm{\Dv f}{L^{p}_x(L^1_\xi)}\leq \Nrm{\Dv f}{L^\infty_{x,\xi}}^{\frac{1}{p'}} \(\iintd \n{\Dv f} \n{\xi}^{3(p-1)} \d x \d\xi\)^{\frac{1}{p}},
	\end{equation*}
	 the boundedness of the operator $\diag(|\Dhv\Weyl{f}|)$ and the propagation of quantum moments $h^3 \tr(|\Dhv\Weyl{f}| \n{\opp}^n)$, for some integer $n$, yield
	 \begin{equation}
		\Nrm{\Diag{\n{\Dhv\Weyl{f}}}}{L^p} \leq C \Nrm{\Dhv\Weyl{f}}{\L^\infty}^{\frac{1}{p'}} \(h^3\Tr{\n{\Dhv\Weyl{f}}\n{\opp}^{3(p-1)}}\)^{\frac{1}{p}}.
	 \end{equation}
	By the Boulkhemair generalization of the Calder\'on--Vaillancourt theorem for the Weyl operators~\cite{boulkhemair_l2_1999}, it holds
	\begin{equation*}
		\Nrm{\Dhv\Weyl{f}}{\L^\infty}\leq C\,\Nrm{\Dv f}{W^{4,\infty}}
	\end{equation*}
	uniformly in $\hbar$, for some $C>0$. Furthermore, by \cite[Proposition~3.3]{lafleche_strong_2021}, the quantum velocity moments are bounded by
	\begin{equation*}
		h^3 \Tr{\n{\Dhv\Weyl{f}}\n{\opp}^{3(p-1)}} \leq C\Nrm{\Dv f}{H_\sigma^\sigma}.
	\end{equation*}
	Gr\"onwall's Lemma concludes the proof.
	

\section{\texorpdfstring{$L^2$}{L2} weak-strong uniqueness criterion}

	In this section we review the proof of~\cite{chong_l2_2022} to obtain an optimal bound in $L^2$ for the semiclassical limit from the Hartree to the Vlasov--Poisson by mimicking the proof of a weak-strong stability criterion for solutions to the Vlasov--Poisson system in $L^2$.

	\subsection{\texorpdfstring{$L^2$}{L2} weak-strong stability criterion  for the Vlasov--Poisson equation}

	\begin{prop}\label{prop:weak-strong-L2}
		Let $f_1$ and $f_2$ be two positive solutions of the Vlasov--Poisson equation verifying { $\Nrm{f_1}{L^\infty}, \Nrm{f_2}{L^\infty} \leq C_{\infty}$}. If 
		\begin{equation*}
			\lambda_{f_2}(t) = \Nrm{{\rho_{f_2}}}{L^{\infty}_x}^{\frac12} \Nrm{\Dv \sqrt{f_2}}{L^3_xL^2_\xi} + C_{\infty}^{\frac12} \Nrm{\Dv \sqrt{f_2}}{L^{3,1}_xL^1_\xi}
		\end{equation*}
		is finite for each $t\in [0,T]$, the following stability estimate holds
		\begin{equation*}
			\BNrm{\sqrt{f_1} - \sqrt{f_2}}{L^2} \leq \BNrm{\sqrt{f_1^\init} - \sqrt{f_2^\init}}{L^2} e^{\Lambda_{f_2}(t)}
		\end{equation*}
		where $\Lambda_{f_2}(t) = C\,\int_0^t \lambda_{f_2}(s)\d s$, for some $C>0$. In particular, 
		\begin{equation*}
			\Nrm{f_1-f_2}{L^2} \leq 2\, C_{\infty}^{\frac12} \Nrm{f_1^\init-f_2^\init}{L^1}^{\frac12} e^{\Lambda_{f_2}(t)}.
		\end{equation*}
	\end{prop}

	\begin{proof}
		Let $v_1=\sqrt{f_1}$, $v_2=\sqrt{f_2}$ and $v := v_1-v_2$. Then $v$ satisfies the equation
		\begin{equation*}
			\big(\dpt+\xi\cdot\Dx+E_{v_1^2}\cdot\Dv\big)v = \big(E_{v_2^2}-E_{v_1^2}\big)\cdot\Dv v_2 = \(\nabla K * \rho_{\(v_2+v_1\)v}\)\cdot\Dv v_2,
		\end{equation*}
		where we use the notations $\rho_g=\int g\,d\xi$ and $E_g=\nabla K*\rho_g$, for any function $g$. The fact that the vector field $(\xi,E_{v_1^2})$ is divergence-free yields
		\begin{equation*}
		\begin{split}
			\frac{1}{2}\frac{\d}{\d t}\Nrm{v}{L^2}^2 &= \iintd v \(\nabla K * \rho_{\(v_2+v_1\)v}\)\cdot\Dv v_2\d x\d\xi
			\\
			&= \iint_{\R^{12}} \big(\n{v(z')}^2 + 2\, v_2(z')\,v(z')\big) \nabla K(x-x') \cdot \(v(z)\, \Dv v_2(z)\) \d z\d z'
			\\
            &=: I_1 + 2 \, I_2.
            \end{split}
		\end{equation*}
        
		\noindent We bound the first term by H\"older's inequality
		\begin{equation*}
		\begin{split}
			I_1 &\leq \BNrm{\nabla K * \intd v\, \Dv v_2 \d\xi}{L^\infty_x} \Nrm{v}{L^2}^2 \leq \BNrm{\n{\nabla K} * \intd \n{\Dv v_2} \d\xi}{L^\infty_x} \Nrm{v}{L^\infty} \Nrm{v}{L^2}^2,
		\end{split}
		\end{equation*}
		and by applying H\"older's inequality for the Lorentz spaces we get 
		\begin{equation*}
			I_1 \leq \Nrm{\nabla K}{L^{\frac{3}{2},\infty}_x} \Nrm{\Dv v_2}{L^{3,1}_x L^1_\xi} \Nrm{v}{L^\infty} \Nrm{v}{L^2}^2.
		\end{equation*}
		To bound $I_2$ we use the Hardy--Littlewood--Sobolev inequality and the Cauchy--Schwarz inequality, to obtain the bound
		\begin{equation*}
		\begin{split}
			I_2 &\leq C\, \BNrm{\intd v\,v_2\d \xi\,}{L^2_x}\,\BNrm{\intd v\,\Dv v_2 \d \xi\,}{L^{\frac{6}{5}}_x}
			\\
			&\leq C \Nrm{\Nrm{v}{L^2_\xi}\Nrm{v_2}{L^2_\xi}}{L^{2}_x} \Nrm{\Nrm{v}{L^2_\xi}\Nrm{\Dv v_2}{L^2_\xi}}{L^{\frac{6}{5}}_x}.
		\end{split}
		\end{equation*}
		By H\"older's inequality, we finally obtain
		\begin{align*}
			I_2 &\leq C \Nrm{v_2}{L^\infty_xL^2_\xi}\Nrm{\Dv v_2}{L^3_xL^2_\xi} \Nrm{v}{L^2}^2 = C \Nrm{{\rho_{f_2}}}{L^\infty_x}^{\frac12} \Nrm{\Dv v_2}{L^3_xL^2_\xi} \Nrm{v}{L^2}^2,
		\end{align*}
		that combined with the bound on $I_1$ leads to
		\begin{equation*}
			\dt \Nrm{v}{L^2}^2 \leq C\, \big(\Nrm{{\rho_{f_2}}}{L^\infty_x}^{{\frac{1}{2}}} \Nrm{\Dv v_2}{L^3_xL^2_\xi} + \Nrm{v}{L^\infty} \Nrm{\Dv v_2}{L^{3,1}_x L^1_\xi}\big) \Nrm{v}{L^2}^2.
		\end{equation*}
	Using that $\Nrm{v}{L^\infty} \leq \Nrm{f_1}{L^\infty}^{\frac{1}{2}} + \Nrm{f_2}{L^\infty}^{\frac{1}{2}} = 2\, C_{\infty}^{\frac{1}{2}}$, Gr\"onwall's lemma concludes the proof.
	\end{proof}

\subsection{Application to the limit from Hartree to Vlasov--Poisson: optimal \texorpdfstring{$L^2$}{L2} bound.}

	\begin{theorem}\label{thm:L2}\cite[Theorem 1.2]{chong_l2_2022}
		Let $\op\geq 0$ be a solution to the Hartree equation~\eqref{eq:Hartree} with initial condition $\op^\init\in\L^1\cap\L^\infty$ and $f\geq 0$, $f\in L^1(\R^6)$ be a solution to the Vlasov--Poisson equation~\eqref{eq:VP} with initial condition $f^\init$ such that
		\begin{equation}\label{eq:hyp-initial}
			f^\init,\,\sqrt{f^\init}\in W^{4,\infty}_4 \cap H^4_4 \quad \text{ and } \quad \iintd f^\init \n{\xi}^{n_1}\d x\d\xi <\infty
		\end{equation}
		for $n_1 > 6$. Then there exist time-dependent functions $\Lambda, C_1, C_2 \in C^0(\R_+,\R_+)$, independent of $\hbar$ and  depending on the initial conditions of equations~\eqref{eq:Hartree} and~\eqref{eq:VP}, such that \begin{equation}\label{eq:Weyl_main_estimate}
			\Nrm{\op-\Weyl{f}}{\L^2} \leq \cC_{\infty}^{\frac12} \(\Nrm{\Wick{\sqrt{f^\init}} - \sqrt{\op^\init}}{\L^2} + C_1(t) \,\hbar\) e^{\Lambda(t)} + C_2(t)\,\hbar.
		\end{equation}
	\end{theorem}

\subsubsection{Proof of Theorem~\ref{thm:L2} part I: auxiliary dynamics}

	We consider the Weyl quantization of~\eqref{eq:VP} and rewrite it as
	\begin{system*}
		i\,\hbar\,\dpt\Weyl{f} &=\com{H_f,\Weyl{f}}-B_f(\Weyl{f})
		\\
		\Weyl{f}(0) &=\Weyl{f^\init},
	\end{system*}
	with $H_f=-\frac{\hbar^2}{2}\Delta+V_f$, being $V_f$ the multiplication operator $V_f(x)=(K*\rho_f)(x)$ and $B$ the operator with kernel~\eqref{eq:B-term}. To mimic the $L^2$ weak-strong uniqueness criterion given in Proposition~\ref{prop:weak-strong-L2}, we want to consider square roots of operators, that are well-defined only for nonnegative operators. For this reason we introduce the Wick quantization  of $\sqrt{f^\init}$ (also called T\"oplitz operator)
	\begin{equation*}
		\Wick{\sqrt{f^\init}} = \Weyl{G_{\hbar}*\sqrt{f^\init}}
	\end{equation*}
	where $G_{\hbar}(z):=(\pi\hbar)^{-3}e^{-\n{z}^2/\hbar}$, $z\in\R^6$, and consider its time-evolution $\widetilde{\op}$ solving the linear Hartree-type equation
	\begin{system}\label{eq:linear-Hartree}
		i\,\hbar\,\dpt\widetilde{\op} &= \com{H_f,\widetilde{\op}\,},
		\\
		\widetilde{\op}^\init &:= \Wick{\sqrt{f^\init}}^2.
	\end{system}
	We therefore bound the difference of $\op$ and $\Weyl{f}$ using the auxiliary dynamics~\eqref{eq:linear-Hartree} as follows
	\begin{equation}\label{eq:L2-bound}
		\Nrm{\op-\Weyl{f}}{\L^2}\leq\Nrm{\op-\widetilde{\op}}{\L^2}+\Nrm{\widetilde{\op}-\Weyl{f}}{\L^2}.
	\end{equation}
	The second term on the right-hand side of~\eqref{eq:L2-bound} measures the error in considering the evolution of the T\"{o}plitz quantization instead of the Weyl quantization. Gr\"{o}nwall's Lemma yields
	\begin{equation}\label{eq:term2-L2}
	\begin{split}
		\Nrm{\widetilde{\op}-\Weyl{f}}{\L^2}
		&\leq\Nrm{\widetilde{\op}^\init-\Weyl{f^\init}}{\L^2}
		\\
		&\qquad +C\hbar\int_0^t\Nrm{\nabla E_f}{L^\infty_x}\Nrm{\nabla_\xi^2 f}{L^2}\d s.
	\end{split}
	\end{equation}
	The second term on the right-hand side of~\eqref{eq:term2-L2} is bounded under the assumptions~\eqref{eq:hyp-initial} on $f^\init$. This follows indeed by the propagation of Sobolev regularity for solutions to the Vlasov--Poisson equation~\eqref{eq:VP} (cf.~\cite{lafleche_strong_2021}). As for the first term on the right-hand side of~\eqref{eq:term2-L2}, we use that
	\begin{align*}
		\Nrm{\widetilde{\op}^\init-\Weyl{f^\init}}{\L^2}
		&=\Nrm{\Wick{\sqrt{f^\init}}^2-\Weyl{f^\init}}{\L^2}
		\\
		&\leq\Nrm{\Wick{\sqrt{f^\init}}^2-\Wick{f^\init}}{\L^2}+\Nrm{\Wick{f^\init}-\Weyl{f^\init}}{\L^2}.
	\end{align*}
	By standard properties of the T\"oplitz quantization (cf.~for instance \cite{lions_sur_1993,chong_l2_2022}) 
	\begin{equation*}
		\Nrm{\Wick{f^\init}-\Weyl{f^\init}}{\L^2}\lesssim\hbar\Nrm{\nabla^2 f^\init}{\L^2},
	\end{equation*}
	whereas $\Nrm{\Wick{\sqrt{f^\init}}^2 - \Wick{f^\init}}{\L^2}$ can be bounded by interpolation using that (cf.~\cite[Lemma 3.1]{chong_l2_2022} for a detailed proof)
	\begin{equation*}
		\Nrm{\Wick{\sqrt{f^\init}}^2-\Wick{f^\init}}{\L^1}\lesssim\hbar\Nrm{\nabla f^\init}{L^2}^2
	\end{equation*}
	and 
	\begin{equation*}
		\ \ \ \Nrm{\Wick{\sqrt{f^\init}}^2-\Wick{f^\init}}{\L^\infty}\lesssim\hbar\Nrm{\nabla f^\init}{L^\infty}^2.
	\end{equation*}
	
	\subsubsection{Proof of Theorem 3.1 part II: $L^2$ approach}
	
	We now focus on the first term on the right-hand side of~\eqref{eq:L2-bound}, where we aim at reproducing  for operators the strategy adopted in Proposition~\ref{prop:weak-strong-L2}. We have
	\begin{equation}\label{eq:L2-main-term}
		\Nrm{\op-\widetilde{\op}}{\L^2}\leq\Nrm{\sqrt{\op}-\sqrt{\widetilde{\op}}}{\L^2}\(\Nrm{\sqrt{\op}}{\L^\infty}+\Nrm{\sqrt{\widetilde{\op}}}{\L^\infty}\).
	\end{equation}
	If $f^\init\geq 0$, then $\Wick{\sqrt{f^\init}} \geq 0$ and therefore its time-evolution given by Equation~\eqref{eq:linear-Hartree} $\widetilde{\op}\geq 0$. Hence the right-hand side of~\eqref{eq:L2-main-term} is well-defined. Furthermore, $\sqrt{\op}$ and $\sqrt{\widetilde{\op}}$ are bounded operators by assumption. Whence, $\Nrm{\op-\widetilde{\op}}{\L^2}\lesssim\Nrm{\sqrt{\op}-\sqrt{\widetilde{\op}}}{\L^2}$.
	We are now in the position of mimicking the $L^2$ weak-strong uniqueness criterion. Performing the time derivative on $(\sqrt{\op}-\sqrt{\widetilde{\op}})$ and  using that if $\op$ solves~\eqref{eq:Hartree} then for $\phi$ smooth function $\phi(\op)$ solves 
	\begin{equation*}
		i\,\hbar\,\dpt\phi(\op)=\com{H_{\op},\phi(\op)},
	\end{equation*}
	Gr\"onwall's Lemma yields
	\begin{equation*}
		\Nrm{\sqrt{\op}-\sqrt{\widetilde{\op}}}{\L^2}\lesssim\Nrm{\sqrt{\op^\init}-\sqrt{\widetilde{\op}^\init}}{\L^2}e^{\Lambda(t)}+\hbar\int_0^t C(s)\,e^{(\Lambda(t)-\Lambda(s))}\d s,
	\end{equation*} 
	with $C(s)$ and $\Lambda(s)$ independent of $\hbar$ and finite under the assumptions on $f^\init$. This latter follows by propagation of regularity for solutions to the Vlasov--Poisson equation~\eqref{eq:VP} (cf.~\cite{lafleche_strong_2021}) and to the linear auxiliary dynamics (cf.~\cite{chong_l2_2022}).

\section{Quantum Optimal transport pseudo-distance}\label{subsec:opt-transport_classical}

	The strategies of Theorem~\ref{thm:LS-20} and Theorem~\ref{thm:L2} are quite general and they generalizes to the relativistic case (cf.~\cite{leopold_propagation_2022}). Application of these methods to other models are certainly possible. Although quite general, the weak-strong type criterion requires some regularity on the solution of the Vlasov--Poisson equation~\eqref{eq:VP}, that is not always compatible with certain classes of quantum states, such as pure states. To treat this important set of initial conditions, one can look at weaker topologies, in which almost no regularity conditions on derivatives of solutions to the Vlasov--Poisson are required.
	
\subsection{Classical optimal transport stability estimate}

	Among the classical metrics on the set of probability measures are the Wasserstein--(Monge--Kantorovich) distances. For $f_1$ and $f_2$ two probability distributions on the phase space $\Rdd$, one defines the set of couplings by $\cC(f_1,f_2)$ as the set of probability measures $\gamma$ on $\Rdd\times\Rdd$ with first marginal $f_1$ and second marginal $f_2$, that is $\intdd \gamma(\cdot,\d y) = f_1$ and $\intdd \gamma(\d x,\cdot) = f_2$. Then, for any $p\geq 1$, the Wasserstein distance of order $p$ is defined by
	\begin{equation*}
		W_p(f_1, f_2) = \(\inf_{\gamma\in\cC(f_1,f_2)} \iintd \n{z_1-z_2}^p \gamma(\d z_1\d z_2)\)^{1/p}.
	\end{equation*}
	It gives a natural distance to compare two solutions of the Vlasov equation. To explain this, let us define the characteristic flow $z(t,z_0) = (x(t,z_0),\xi(t,z_0))$ associated to a solution $f$ of the Vlasov equation by $z(0,z_0) = z_0$ and
	\begin{system*}
		&\dpt x(t,z_0) = \xi(t,z_0)
		\\
		&\dpt\xi(t,z_0) = E_f(x(t,z_0))
	\end{system*}
	and it verifies for any bounded continuous function $\varphi$
	\begin{equation*}
		\intdd \varphi(z) \,f(t,z)\d z = \intdd \varphi(x(t,z_0),\xi(t,z_0)) \,f^\init(z_0) \d z_0.
	\end{equation*}
	Then, if $E_{f_1}$ and $E_{f_2}$ are sufficiently close and regular, one observes that two initially close characteristic trajectories $z_1 = z_1(t,z_0)$ following the flow of $f_1$ and $z_2 = z_2(t,z'_0)$ following the flow of $f_2$ will remain close since
	\begin{align*}
		\dpt \n{z_1-z_2}^2 &= 2\(x_1-x_2\)\cdot\(\xi_1-\xi_2\) + 2\(\xi_1-\xi_2\)\cdot\(E_{f_1}(x_1) - E_{f_2}(x_2)\)
		\\
		&\leq \(1+\omega_{f_1}(x_1,x_2)^2\)\n{x_1-x_2}^2 + 3 \n{\xi_1-\xi_2}^2 + \n{E_{f_1}(x_2)-E_{f_2}(x_2)}^2
	\end{align*}
	where 
	\begin{equation*}
		\omega_{f_1}(x_1,x_2) = \sup_{x_1,x_2} \frac{\n{E_{f_1}(x_1) - E_{f_1}(x_2)}}{\n{x_1-x_2}}	
	\end{equation*}
	so that one should expect that $\n{z_1(t)-z_2(t)}$ does not grow too fast. Hence, to control an initially small quantity of the form
	\begin{equation}\label{eq:test_transport}
		\Eps_p = \iintd \n{z_1-z_2}^p \gamma(t,\d z_1\d z_2)
	\end{equation}
	one can choose $\gamma$ such that
	\begin{equation*}
		\iintdd \varphi(z_1,z_2) \,\gamma(t,z_1,z_2)\d z_1\d z_2 = \iintdd \varphi(z_1(t,z),z_2(t,z')) \,\gamma^\init(z,z')\d z\d z'
	\end{equation*}
	or equivalently, $\gamma$ solving
	\begin{equation*}
		\dpt \gamma + \(\xi_1\cdot\nabla_{x_1}+E_{f_1}\cdot\nabla_{\xi_1}\)\gamma  +\(\xi_2\cdot\nabla_{x_2}+E_{f_2}\cdot\nabla_{\xi_2}\)\gamma = 0.
	\end{equation*}
	with initial condition $\gamma(0,\cdot,\cdot) = \gamma^\init(\cdot,\cdot)$. Taking the initial coupling $\gamma^\init$ to be the optimal coupling then ensures that~\eqref{eq:test_transport} is initially small if $f_1$ and $f_2$ are initially close in $W_p$. In the case of the Coulomb potential and when $\rho_{f_1}$ and $\rho_{f_2}$ are uniformly bounded, then $E_{f_1}$ is log-Lipshitz, that is $\omega_{f_1}(x_1,x_2) \leq C\max\!\(1,-\ln(\n{x_1-x_2})\)$, and it was proved by Loeper~\cite{loeper_uniqueness_2006} that $\n{E_{f_1}(x_2)-E_{f_2}(x_2)}$ is controlled by $W_2(f_1,f_2)$. As proved in \cite{loeper_uniqueness_2006}, this leads to an inequality of the form
	\begin{equation}\label{eq:Gronwall_W2}
		\dpt \Eps_2(t) \leq C(t)\,\Eps_2(t) \max\!\(1,-\ln(\Eps_2(t))\).
	\end{equation}
	where $C(t)>0$ depends on the growth of $\Nrm{\rho_{f_1}}{L^\infty}+\Nrm{\rho_{f_2}}{L^\infty}$. One then concludes by Gronwall's lemma to a stability estimate of the form
	\begin{equation*}
		W_2(f_1,f_2) \leq W_2(f_1^\init,f_2^\init)^{e^{\theta\lambda(t)}}\,e^{e^{\lambda(t)}}
	\end{equation*}
	where $\theta = \sign(\ln W_2(f_1,f_2))$ and $\lambda(t)>0$ depends on $C(t)$, and this implies uniqueness for the Vlasov--Poisson equation. Let us mention that using a different weight for the phase space variables (see~\cite{iacobelli_new_2022}), one can also obtain an inequality with a square root of the logarithm instead of a logarithm in Inequality~\eqref{eq:Gronwall_W2}.

\subsection{Quantum optimal transport}

	To obtain an application to semiclassical estimates, one can use the semiclassical analogue of the optimal transport distance defined by Golse and Paul~\cite{golse_schrodinger_2017}. Let $\cP$ be the set of density operators, that is the set of positive trace class operators verifying $h^3\Tr{\op}=1$. As in the classical case, one first introduces the notion of a coupling between a density operator $\op$ and a classical kinetic density $f$ as an operator valued function $\opgam\in L^1(\Rdd,\cP)$ such that
	\begin{align*}
		h^3\Tr{\opgam(z)} = f(z), \qquad \text{ and }  \qquad\intdd \opgam(z)\d z = \op.
	\end{align*}
	Denoting by $\cC(f,\op)$ the set of semiclassical couplings, the semiclassical optimal transport pseudometric of order $2$ is defined by
	\begin{equation}\label{def:Wh}
		\Wh(f,\op) := \(\inf_{\opgam\in\cC(f,\op)}\intdd h^3\Tr{\mathbf{c}_\hbar(z)\opgam(z)}\d z\)^\frac{1}{2},
	\end{equation}
	where for any $z\in\Rdd$, $\opc_\hbar(z)$ is an operator defined by it action on functions $\varphi$ of the form $\opc_\hbar(z)\varphi(y) = (\n{x-y}^2 + \n{\xi-\opp}^2)\varphi(y)$, with $z = (x,\xi)$ and $\opp = -i\hbar\nabla_y$. This is not a distance, as it verifies for example $\Wh(f,\op)^2 \geq 3\,\hbar$, but it is comparable to the classical Wasserstein distance $W_2$ between the Husimi transform $\tilde{f}_{\op}$ of the quantum density operator and the normal kinetic density $f$ in the sense that it verifies~\cite{golse_schrodinger_2017}
	\begin{equation}\label{eq:comparaison_W2}
		W_2(f,\tilde{f}_{\op})^2 \leq \Wh(f,\op)^2 + 3\,\hbar,
	\end{equation}
	and in the special case when $\op$ is the Wick quantization of some function $g$,
	\begin{equation}\label{eq:comparaison_W2_2}
		\Wh(f,\Wick{g})^2 \leq W_2(f,g)^2 + 3\,\hbar.
	\end{equation}
	Moreover, it verifies approximate triangle inequalities~\cite{golse_quantum_2021}. For general density matrices, it can be proved that $\Wh(f,\op) \leq W_2(f,\tilde{f}_{\op}) + \sqrt{3\hbar} + \hbar \,\norm{\Dh\sqrt{\op}}_{\L^2}$ (see~\cite{lafleche_quantum_2023}). In particular, this allows to treat more singular states such as some classes of projection operators verifying $\op^2 = \op \geq 0$ for which in general one can obtains bounds of the form $\hbar\,\norm{\Dh\sqrt{\op}}_{\L^2} = \hbar\,\norm{\Dh\op}_{\L^2} \leq C\,\hbar^\theta$ for some $\theta\in[0,1/2]$ (see \cite{lafleche_optimal_2023}).
	
\subsection{Application to the semiclassical limit}

	The same strategy as in the classical case was used in~\cite{lafleche_propagation_2019, lafleche_global_2021} to prove the following theorem.
	\begin{theorem}\cite[Theorem~5]{lafleche_propagation_2019}\label{th:CV_VP}
		Let $f$ be a solution of the Vlasov--Poisson equation~\eqref{eq:VP} and $\op$ be a solution of the Hartree equation~\eqref{eq:Hartree} equation with respective initial conditions $f^\init$ and $\op^\init$ verifying
		\begin{align*}
			f^\init \in \P \cap L^1_{n_0}\cap L^\infty_n \qquad \text{ and } \op^\init \in \cP\cap \L^1_{16}\cap\L^\infty_4.
		\end{align*}
		uniformly in $\hbar$ for some $(n_0,n)\in\R_+^2$ verifying $n_0 > 6$ and $n>3$. Then there exists $T>0$ such that
		\begin{equation}\label{eq:propag_moments}
			h^3\Tr{\n{\opp}^{16}\op} \in L^\infty([0,T]) \qquad \text{ and } \qquad\op \in L^\infty([0,T],\L^\infty_4),
		\end{equation}
		uniformly in $\hbar$, and there exists a constant $C_T$ depending only on the initial conditions and independent of $\hbar$ such that
		\begin{equation}\label{eq:Wh_stability}
			\Wh(f,\op) \leq \Wh(f^\init,\op^\init)^{e^{\theta\lambda(t)}} \,e^{e^{\lambda(t)}}
		\end{equation}
		where $\theta = \sign(\ln \Wh(f,\op))$ and $\lambda(t)\in C^0([0,T],\R_+)$ is an increasing function of time depending on the growth of $\Nrm{\rho_f}{L^\infty(\Rd)} + \Nrm{\rho}{L^\infty(\Rd)}$ and verifying $\lambda(0) = 0$.
	\end{theorem}
	
	The function $\lambda(t)$ is detailed in \cite[Theorem~5]{lafleche_global_2021}. The above theorem implies in particular that if initially there exists a constant $C$ independent of $\hbar$ such that $\Wh(f^\init,\op^\init) \leq C\sqrt{\hbar} \leq 1$, then
	\begin{equation*}
		W_2(f,\tilde{f}_{\op}) \leq \hbar^{e^{-\lambda(t)}} \,e^{e^{\lambda(t)}}.
	\end{equation*}
	
	\begin{proof}[Steps of the proof]
		Several steps are needed to get the complete result. As in the classical case, the stability estimate in the optimal transport metric uses the fact that both $\rho_f$ and $\rho$ are bounded (uniformly in $\hbar$). To get this property for the Hartree equation, one proceeds in three steps.
		
		\begin{itemize}
		\item First, one proves the propagation of moments of the form $h^3\Tr{\n{\opp}^n\op}$ with $n\in 2\N$ \cite[Theorem~3]{lafleche_propagation_2019}. The proof uses the fact that moments of order two are bounded by the energy conservation, the fact that $\Nrm{\op}{\L^1} + \Nrm{\op}{\L^\infty}$ is conserved and Lieb--Thirring type inequalities. This is at this step that only a local in time propagation is proved, which prevents the theorem to hold for all times $T>0$.
		\item The second point is to prove the propagation of weighted Schatten norms of the form $\Nrm{\op\,\n{\opp}^n}{\L^p}$ uniformly with respect to $p$ and $\hbar$~\cite[Proposition~5.3]{lafleche_propagation_2019} using the previously proved boundedness of moments.
		\item The last point consists in taking $p\to\infty$ to get that $\Nrm{\op\,\(1+\n{\opp}^n\)}{\L^\infty}$ is bounded uniformly in $\hbar$, from which one deduces that $\Nrm{\rho}{L^\infty}$ is bounded uniformly in $\hbar$.
		\end{itemize}
		Once these points are proved, the proof of the theorem follows as in~\cite{golse_schrodinger_2017} by adapting the classical proof of stability for the Vlasov equation but replacing one of the solutions by a solution of the Hartree equation. More precisely, for any coupling $\opgam^\init\in \cC(f^\init,\op^\init)$, one looks at the solution $\opgam = \opgam(t,z)$ to
		\begin{equation*}
			\dpt \opgam + \(\xi\cdot\Dx + E_f\cdot\Dv\)\opgam - \frac{1}{i\hbar} \com{H_{\op},\opgam}
		\end{equation*} 
		and obtains a Gr\"onwall's type inequality on (the logarithm of)
		\begin{equation*}
			\Eps_{2,\hbar}(t) := \intdd h^3\Tr{\mathbf{c}_\hbar(z)\opgam(t,z)}\d z.
		\end{equation*}
		Since $\Wh(f,\op)\leq \Eps_{2,\hbar}$, the result follows by minimizing over all couplings $\opgam^\init$.
	\end{proof}

{\bf Acknowledgments.} C.S. acknowledges the NCCR SwissMAP and the support of the SNSF through the Eccellenza project PCEFP2\_181153. L.L. has received funding from the European Research Council (ERC) under the European Union’s Horizon 2020 research and innovation program (grant agreement No 865711).


\bibliographystyle{abbrv} 
\bibliography{Vlasov}

\end{document}